\documentclass[11pt,reqno]{amsart}
\usepackage[cp1251]{inputenc}
\usepackage[english]{babel}
\usepackage{amsmath}
\usepackage{amssymb}
\usepackage{amsfonts}
\usepackage{graphicx}
\usepackage{xcolor}
\usepackage[colorlinks]{hyperref}
\usepackage{array}

\newtheorem{theorem}{Theorem}

\newtheorem{corollary}{Corollary}
\newtheorem{proposition}{Proposition}

\begin{document}

\title[Nonstrictly hyperbolic systems]{
Nonstrictly hyperbolic systems and their application to study of Euler-Poisson equations
}

\author{Marko K. Turzynsky}

\address{ Department of Digital Technologies for Control of Transport Processes, Russian University of Transport,
Obrastsova str., 9b9, Moscow, 127004, Russian Federation, M13041@yandex.ru}

\subjclass{Primary 35Q60; Secondary 35L60, 35L67, 34M10}

\keywords{Nonstrictly hyperbolic systems, Euler-Poisson equations, singularity formation, blow-up}

\maketitle

\begin{abstract}
We study inhomogeneous non-strictly hyperbolic systems of two equations, which are a formal generalization of the transformed one-dimensional Euler-Poisson equations. For such systems, a complete classification of the behavior of the solution is carried out depending on the right side. In particular, criteria for the formation of singularities in the solution of the Cauchy problem in terms of initial data are found. Also we determine the domains of attraction of equilibria of the extended system for derivatives. We prove the existence of solutions in the form of simple waves. The results obtained are applied to study the main model cases of the Euler-Poisson equations.
\end{abstract}

\section{Introduction}

The pressureless Euler-Poisson equations, which have important applications in astrophysics and semiconductor physics, in the one-dimensional case have the form
\begin{equation}\label{ep}
n_t + nn_x = 0,\quad V_t + V V_x = - k \Phi_x - \gamma V,\quad \Delta \Phi=n-N,
\end{equation}
where $n(t,x)$ -- density, $V(t,x)$ -- velocity, $\Phi(t,x)$ -- potential of some force.
Here $k$ -- a constant which determines whether this potential is repulsive ($k=-1$) or attractive ($k=1$); $N$ -- constant normalized background density equal to $0$ or $1$ depending on the model; $\gamma\ge 0$ -- constant coefficient of friction, which, for example, can be interpreted as a coefficient of linear friction. Such equations were considered in \cite{Tad}, where in terms of the initial data
\begin{equation}\label{ep0}
(n(0,x), V(0,x))=(n_0(x), V_0(x))\in C^1(\mathbb{R})
\end{equation}
we obtained exact conditions under which the solution of the Cauchy problem loses smoothness in a finite time. The methods for obtaining such conditions were very cumbersome and required individual study for each set of coefficients.

Much later, in \cite{Ros2}, a similar problem was solved for the cold plasma equations that can be obtained from the system \eqref{ep} for $k=-1$, $N=1$, $\gamma=0$. It turned out that after transforming our system, the methods to study it are significantly simplified due to the possibility of reducing it to a system of ordinary differential equations. Subsequently, a technique based on Radon's theorem was developed, which made it possible to successfully analyze systems associated with the original cold plasma system. These are, for example, plasma with electron-ion collisions (in fact, the case of $\gamma>0$) \cite{Ros5}, plasma in an external magnetic field \cite{Ros2}, \cite{Ros3}, a multidimensional axisymmetric model of cold plasma \cite{Ros1}.

Let us carry out the procedure of reducing the system \eqref{ep} to a non-strictly hyperbolic system of quasilinear equations.
To do this, we introduce a new unknown function $E=-\Phi_x$, we obtain $n=N - E_x$ from the third equation \eqref{ep} and substitute this equality into the first equation \eqref{ep}. We obtain the system
\begin{equation}\label{ep1}
V_t + V V_x = kE - \gamma V,\quad E_t + VE_x = N V
\end{equation}
with initial conditions
\begin{equation}\label{gg0}
(V(0,x), E(0,x))=(V_0(x), E_0(x)),
\end{equation}
where $E_0$ is related to $n_0$ by the condition $(E_0)_x = N - n_0(x)$.

The system \eqref{ep1} can be written in an equivalent form, introducing shortly ${\bf V}=(V_1,\,V_2)^T$, $V_1=V, \, V_2=E$ as
\begin{equation}\label{gg}
\frac{\partial {\bf V}}{\partial t} + V_1 \mathbb{E}\frac{\partial {\bf V}}{\partial x}=Q {\bf V},
\end{equation}
where  $\mathbb{E}$ is identity $2$ x $2$ matrix,
$Q=\left(\begin{array}{ccr} -\gamma & k\\
N & 0 \end{array}\right)$ -- constant $2$ x $2$ matrix.
Many important models of physics and hemodynamics (for example, \cite{Ros6}) can be reduced to the system \eqref{gg} .

The system \eqref{gg} is non-strictly hyperbolic and posesses a locally in time solution that has the same smoothness as the initial condition \cite{Daf}. Note that \eqref{gg} is a special case of a non-strictly hyperbolic system of two equations for ${\bf V}=(V_1,\, V_2)^T$
\begin{equation}\label{ggg}
\frac{\partial {\bf V}}{\partial t} + V_1 \mathbb{E}\frac{\partial {\bf V}}{\partial x}=Q {\bf V}
\end{equation}
with arbitrary constant matrix $Q$, where $Q=\begin{smallmatrix}\left(\begin{array}{ccr} a & b\\
c & d \end{array}\right)\end{smallmatrix}$. We will consider initial data of the following form
$${\bf V}|_{t=0}={\bf V}_0(x)
\in C^2({\mathbb R}).$$
The increased smoothness of the initial data is due to the fact that the proof technique requires consideration of an extended system of derivatives.

In this paper we show that there is a standard procedure to obtain criteria for the formation of solutions' singularities of the Cauchy problem for \eqref{ggg} in terms of the initial data. This procedure consists in constructing a quadratically nonlinear extended system for derivatives and linearizing it using Radon's theorem (Theorem~\ref{t1} of Section 2, see \cite{Frei} for more details, \cite{Reid}). Using the possibility of getting an analytical solution of the extended system, in Section 3 we carry out a complete classification of the behavior of the solution to the Cauchy problem depending on the matrix $Q$. In Section 4 we study the stability of equilibria of the extended system.

In Section 5 we show that all systems of the form \eqref{ggg} have subclasses of solutions of the simple wave type. In particular, for the Euler-Poisson equations these subclasses can be found explicitly and the system reduces to one equation. In Section 5 we study the equilibria of the extended system for derivatives and classify their stability. We show that the initial data corresponding to globally smooth solutions of the Cauchy problem possess the property that at each point of the real axis a pair of derivatives of the initial data (initial data of the extended system) fall into the domain of attraction of the final singular point. The criteria for the formation of singularities obtained in Section 3 allow us to accurately describe domains of attraction of equilibria.

In the final Section 6, we apply the results obtained to the original system of Euler-Poisson equations \eqref{ep} for the main model cases and show that the criteria from the paper \cite{Tad} can be obtained in a much simpler and more universal way. In addition, we study the asymptotic behavior of the solution for large $t$, and graphically depict the regions where we need to place the derivatives of the initial data for the solution to preserve its global smoothness (for $\gamma=0$).

\section{Criteria for the formation of singularities in terms of decisive function}

The main advantage to consider the system \eqref{ggg} is the possibility of obtaining its solution along the characteristics $\dot x(t) = V$, that means to reduce the problem to a solution of a system of ordinary differential equations. Thus, the components of the solution satisfy the linear system
\begin{equation}\label{sy}
\frac{d{\bf V}}{dt}=Q{\bf V},\quad \frac{dx}{dt}=V
\end{equation}
with initial condition
\begin{equation}\label{sy0}
{\bf V}(0)={\bf V}_0(x_0),\quad x(0)=x_0,\quad x_0\in\mathbb{R}.
\end{equation}

If we differentiate the system \eqref{ggg} with respect to x, then we obtain an extended quadratic nonlinear system of derivatives ${\bf v}=(v_1,\,v_2)$, $v_1=V_x$, $v_2=E_x$ along the characteristics $\frac{dx}{dt}=V$, $x(0)=x_0$:
\begin{equation}\label{1}
\frac{d\bf v}{dt}=-v_1{\bf v}+Q{\bf v}.
\end{equation}

\noindent For further study in order to simplify the system \eqref{1}, we reduce the matrix $Q$ to the Jordan normal form. It is well known that there exists a non-singular matrix $A=\left(\begin{smallmatrix} a_{11} & a_{12} \\ a_{21} & a_{22}\end{smallmatrix}\right)$ such that $AQA^{-1}=J$, where $J$ is the Jordan normal form of the matrix $Q$. Thus, introducing new unknown functions
\begin{equation}\label{vec}
{\bf w}=A{\bf v},\quad {\bf w}=(w_1,\,w_2)^T
\end{equation}
we can rewrite the system \eqref{1} as
\begin{equation}\label{ss}
\dot {\bf w} = -v_1 {\bf w} + J {\bf w},
\end{equation}
where $v_1 = \tfrac{1}{\det{A}}(a_{22} w_1-a_{12} w_2)$.

The main component of further research is the following theorem (\cite{Frei}, \cite{Reid}):

\begin{theorem}\label{t1} (Radon's lemma)
Matrix Riccati equation

\begin{equation}\label{Ricc}
\dot W = M_{21}(t)+M_{22}(t)W-WM_{11}(t)-WM_{12}(t)W
\end{equation}

\noindent ($W=W(t)$ -- $n\times m$ matrix, $M_{21}$ -- $n\times m$ matrix, $M_{22}$ -- $m\times m$ matrix, $M_{11}$ -- $n\times n$ matrix, $M_{12}$ -- $m\times n$ matrix) is equivalent to homogeneous linear matrix equation

\begin{equation}\label{hle}
\dot Y = M(t)Y,\quad M=\left(\begin{array}{ccr} M_{11} & M_{12}\\
M_{21} & M_{22} \end{array}\right)
\end{equation}

\noindent ($Y=Y(t)$ -- $n\times (n+m)$ matrix, $M$ -- $(n+m)\times (n+m)$ matrix) in the following sense.

Let on some interval $I\in\mathbb{R}$ the matrix function $Y(t)=\left(\begin{array}{ccr} R(t)\\ S(t) \end{array}\right)$
($R(t)$ is $n\times n$ matrix, $S(t)$ is $n\times m$ matrix) be a solution of the system \eqref{hle} with the initial condition

$$Y(0)=\left(\begin{array}{ccr} \mathbb{E} \\
W_0 \end{array}\right)$$

\noindent ($\mathbb{E}$ -- identity $n\times n$ matrix, $W_0$ -- constant $n\times m$ matrix)
and $\det R \neq 0$. Then $W(t)=S(t)R^{-1}(t)$ is a solution of \eqref{Ricc} with the initial condition $W(0)=W_0$ at $I$.

\end{theorem}

To apply this theorem, we rewrite the system \eqref{ss} in matrix form \eqref{Ricc}, then

$$W={\bf w},\quad M_{11}=(0),\quad M_{12}=(\tfrac{a_{22}}{\det{A}}\,\,-\tfrac{a_{12}}{\det{A}}),\quad M_{21}=(0\,\,0)^T,\quad M_{22}=J.$$

\noindent Thus, we obtain the Cauchy problem for a linear system on matrices $R(t)=(q(t))$, $S(t)=(u_1(t)\,\,u_2(t))^T$:

\begin{eqnarray}\label{2}
\left(\begin{array}{ccc} \dot q \\ \dot u_1 \\ \dot u_2\end{array}\right)=\left(\begin{array}{ccc} 0 & \tfrac{a_{22}}{\det{A}} & -\tfrac{a_{12}}{\det{A}}\\
0& j_{11} & j_{12}\\ 0 &j_{21} &j_{22} \end{array}\right) \left(\begin{array}{ccr} q \\ u_1 \\ u_2\end{array}\right),\quad \left(\begin{array}{ccr} q(0) \\ u_1(0) \\ u_2(0)\end{array}\right)=\left(\begin{array}{ccr} 1 \\ v_1(0) \\ v_2(0)\end{array}\right),
\end{eqnarray}

\noindent where $J=(j_{ik})$, $i=1,2$, $k=1,2$; $v_1(0)=V_x(x_0)$, $v_2(0)=E_x(x_0)$, $x_0\in\mathbb{R}$.

We note that the formation of singularity for hyperbolic systems means that either the solution or its first derivatives \cite{Daf} tend to infinity. We see that the solution described by the linear system \eqref{sy} with the initial condition \eqref{sy0} does not form a singularity. Further, it follows from Radon's theorem that ${\bf w}(t)=(\frac{u_1(t)}{q(t)},\frac{u_2(t)}{q(t)})$. This means that the derivatives of the solution have a singularity if and only if $q(t)$ which is a part of the solution of the system \eqref{2} vanishes on the positive semiaxis (we will further call $q(t)$ the {\it decisive function}). We can formulate this result as the following theorem:
\begin{theorem}\label{t2}
1. The behavior of the solution of the Cauchy problem \eqref{ep1}, \eqref{gg0} along the characteristics can be completely described by the dynamics of the linear systems \eqref{sy} and \eqref{2}.

2. The solution of the Cauchy problem \eqref{ep1}, \eqref{gg0} remains smooth for all $t>0$ if and only if for each point $x_0\in\mathbb{R}$ the decisive function $q(t)$ does not vanish for all $t>0$.

3. If there exists a point $x_0\in\mathbb{R}$ such that
for some finite value $t_*>0$ we have $q(t_*)$=0, then the solution of the Cauchy problem \eqref{ep1}, \eqref{gg0} has a singularity in finite time. The exact moment of time $T_*$ when we have the singularity can be found using the formula $T_*=\inf\limits_{x_0\in\mathbb{R}} \{t_*>0|\,q(t_*)=0\}$.
\end{theorem}

System \eqref{2} is a linear homogeneous system with constant coefficients, so it can be solved explicitly, which allows us to obtain necessary and sufficient conditions for the formation of singularities of the solution in a more explicit form than in Theorem \ref{t2}. In the next section, we will investigate in which situations we have / do not have intersections of the graph of function $q(t)$ with $Ot$ axis for $t>0$ depending on the matrix $Q$ and thereby obtain a complete classification of the possible behavior of the solution for all possible matrices $Q$. 

\section{Formulation of theorems of singularity formation in terms of the initial system}
\subsection{Obtaining the components of the solution for the system \eqref{2}}

In the previous section we proceeded from the matrix $Q$ to its Jordan form $J$ by making the transformation $J=AQA^{-1}$. Obviously, to find the solution $q(t)$ of the system \eqref{2}, we first need to resolve the system
\begin{equation}\label{rr}
\dot {\bf u}=J{\bf u},\quad {\bf u}=(u_1,\,u_2)^T
\end{equation}
and then find $q(t)$ by integrating the equation $\dot q=\tfrac{a_{22}}{\det{A}}u_1-
\tfrac{a_{12}}{\det{A}}u_2$. Jordan forms of the matrix $Q$ can only be of three different forms, which will simplify our further investigation.

It is well known that every $2\times 2$ matrix $Q$ can be reduced to one of the following forms by a non-degenerate linear transformation:

\begin{itemize}
\item[$\mathbb A$.] $J_{\mathbb A}=\left(\begin{array}{ccr} \lambda_1 & 0\\
0 & \lambda_2 \end{array}\right)$, $\lambda_i\in\mathbb{R}$, $\lambda_1\ge \lambda_2$;

\item[$\mathbb B$.]  $J_{\mathbb B}=\left(\begin{array}{ccr} \lambda & 1\\
0 & \lambda \end{array}\right)$, $\lambda\in\mathbb{R}$;

\item[$\mathbb C$.]  $J_{\mathbb C}=\left(\begin{array}{ccr} \alpha & \beta\\
-\beta & \alpha \end{array}\right)$, $\alpha,\beta\in\mathbb{R}$, $\beta>0$.
\end{itemize}



\noindent Thus, we obtain the following proposition, in which we find exact solutions $u_1(t)$, $u_2(t)$ of the system \eqref{rr} depending on the form of $J$ (its proof is omitted due to its obviousness).

\begin{proposition}
Let us introduce: $\lambda_1$, $\lambda_2$ -- eigenvalues of the matrix $Q$, which satisfy the equation $\lambda^2-(a+d)\lambda+(ad-bc)=0$.

\begin{itemize}
\item[1.] Let $\lambda_1$ and $\lambda_2$ be real, and each of them has the same number of independent eigenvectors as the multiplicity of this eigenvalue. Then the solution of the system \eqref{rr} has the form $u_1(t)=\tilde C_1 \exp(\lambda_1 t)$, $u_2(t)=\tilde C_2 \exp(\lambda_2 t)$.

\item[2.] If $\lambda_1$ and $\lambda_2$ are two real numbers, which coincide, and there exists only one eigenvector, then the solution of the system \eqref{rr} has the form $u_1(t)=(\tilde C_1+\tilde C_2 t)\exp(\lambda t)$, $u_2(t)=\tilde C_2 \exp{\lambda t}$, where $\lambda=\lambda_i$, $i=1,2$.

\item[3.] If $\lambda_1$ and $\lambda_2$ are complex conjugate eigenvalues, then the solution of the system \eqref{rr} has the form $u_1(t)=\exp(\alpha t)$ $(\tilde C_1\sin\beta t+\tilde C_2\cos\beta t),$ $u_2(t)=\exp(\alpha t)(-\tilde C_2\sin\beta t+C_1\cos\beta t),$ where $\alpha = \Re(\lambda_1)$, $\beta=\Im(\lambda_1)$ and $\Im(\lambda_1)>0$.
\end{itemize}

The constants $\tilde C_1$ and $\tilde C_2$ are related to the initial data in items 1 and 2 of the Proposition as
$$(\tilde C_1\,\, \tilde C_2)=(v_1(0)\,\, v_2(0))A^T.$$

The constants $\tilde C_1$ and $\tilde C_2$  are related to the initial data in item 3 as
$$(\tilde C_2\,\, \tilde C_1)=(v_1(0)\,\, v_2(0))A^T.$$
\end{proposition}

\subsection{Case ${\mathbb A}$ ($J_{\mathbb A}$ is a diagonal matrix)} We define constants $C_1$, $C_2$ in terms of initial data $v_1(0)$ and $v_2(0)$ as follows:
\begin{equation}\label{fr1}
C_1=\frac{a_{22}}{\det A}(a_{11}v_1(0)+a_{12}v_2(0)),\quad C_2=-\frac{a_{12}}{\det A}(a_{21}v_1(0)+a_{22}v_2(0)).
\end{equation}

\noindent Let us note that the constants here and below, in paragraphs 3.3 and 3.4, are defined uniquely. In order to find the function $q(t)=\int u_1(t)\,dt=\int (C_1 \exp(\lambda_1 t)+C_2 \exp(\lambda_2 t))$ $dt$, $q(0)=1$,
we need to consider three different situations:

\begin{itemize}
\item[1.] if $\lambda_1 \neq 0$, $\lambda_2 \neq 0$, we get
$$q(t)=\frac{C_1}{\lambda_1}(\exp(\lambda_1 t)-1)+
\frac{C_2}{\lambda_2}(\exp(\lambda_2 t)-1)+1;$$

\item[2.] if $\lambda_1 = 0$, $\lambda_2 \neq 0$, we get
$$q(t)=C_1 t + \frac{C_2}{\lambda_2}(\exp(\lambda_2 t)-1)+1;$$

\item[3.] if $\lambda_1 \neq 0$, $\lambda_2 = 0$, we get
$$q(t)=\frac{C_1}{\lambda_1}(\exp(\lambda_1 t)-1)+ C_2 t+1.$$
\end{itemize}

We will now prove successively four corollaries of Theorem \ref{t2} for the situations described above. We note also that we do not need to consider the case $3$ separately, since it can be reduced to case $2$ by replacing $\lambda_1$ with $\lambda_2$. In Corollary \ref{t31} we consider the situation of different but nonzero eigenvalues; in Corollary \ref{t41} we have the situation of equal and nonzero eigenvalues; in Corollaries \ref{t4} and \ref{t5} we analyze the situation of one zero and one nonzero eigenvalue.

\begin{corollary}\label{t31}
We suppose $\lambda_1\neq \lambda_2$, $\lambda_1 \neq 0$, $\lambda_2\neq 0$ and define constants $C_1$, $C_2$ by condition \eqref{fr1}. Then the solution of system \eqref{1} with initial data $(v_1(0), v_2(0))$ loses smoothness in finite time if and only if one of the following situations occurs:

\begin{itemize}
\item[1.] in case of $\lambda_1>0$

\begin{itemize}
\item[a.] $C_1<0$;

\item[b.] $C_1>0$, $C_2<0$, $C_1+C_2<0$, $C_1 (-C_1^{-1} C_2)^{\tfrac{\lambda
_1}{\lambda_1-\lambda_2}} (\lambda_1^{-1}-\lambda_2^{-1})+1-
\frac{C_1}{\lambda_1}-\frac{C_2}{\lambda_2} \le 0$;

\item[c.] $C_1=0$, $C_2<0$, $\lambda_2>0$;

\item[d.] $C_1=0$, $C_2<\lambda_2$, $\lambda_2<0$;
\end{itemize}

\item[2.] in case of $\lambda_1<0$

\begin{itemize}
\item[a.] $C_1>0$, $C_2<0$, $C_1+C_2<0$, $C_1 (-C_1^{-1} C_2)^{\tfrac{\lambda
_1}{\lambda_1-\lambda_2}} (\lambda_1^{-1}-\lambda_2^{-1})+1-
\frac{C_1}{\lambda_1}-\frac{C_2}{\lambda_2} \le 0$;

\item[b.] $C_1<0$, $\frac{C_1}{\lambda_1}+\frac{C_2}{\lambda_2}>1$;

\item[c.] $C_1=0$, $C_2<0$, $\lambda_2>0$;

\item[d.] $C_1=0$, $C_2<\lambda_2$, $\lambda_2<0$.
\end{itemize}
\end{itemize}
\end{corollary}

\begin{proof}Let us consider the cases one by one.
Let us first note that
\begin{equation}\label{pro1}
q'(t)=C_1\exp(\lambda_1 t)+C_2\exp(\lambda_2 t),\quad q'(0)=C_1+C_2.
\end{equation}

I. We start from the situation when $\lambda_1>0$ (and neither of the constants $C_1$ or $C_2$ is zero). In this case we also have
$q(t) \sim \frac{C_1}{\lambda_1}\exp(\lambda_1 t)$ as $t\to +\infty$.

1. If both constants $C_1>0$ and $C_2>0$, then $q'(t)>0$ by virtue of the first equality \eqref{pro1} and the function $q(t)$ increases at $t>0$. There are no intersection points with the positive semi-axis.

2. Let $C_1>0$, $C_2<0$. Then the function $q(t)$ necessarily has an extremum point $t_{extr}$, which is found from the equality $q'(t)=0$, whence $t_{extr}=\frac{1}{\lambda_1-\lambda_2}\ln\left(-\frac{C_2}{C_1}\right)$.
We see that $t_{extr}>0$ if and only if $C_1+C_2<0$. Otherwise, when $C_1+C_2 \ge 0$ and
the extremum point is not on the positive semiaxis, we see from the second equality \eqref{pro1}  that $q'(0)$ is non-negative and hence the function $q(t)$ increases for $t>0$. There are no intersection points of its graph with the $Ot$ axis, $t>0$. Let us return to the situation when $C_1+C_2<0$. In this case the extremum point found is a minimum point since $q(t)\to +\infty$ at $t\to +\infty$. The value of the function at this minimum point is equal to
\begin{equation}\label{tex}
q(t_{extr})=
C_1 (-C_1^{-1} C_2)^{\tfrac{\lambda
_1}{\lambda_1-\lambda_2}} (\lambda_1^{-1}-\lambda_2^{-1})+1-
\frac{C_1}{\lambda_1}-\frac{C_2}{\lambda_2},
\end{equation}and it must be non-positive for a point of intersection with the $Ot$ axis at $t>0$ to exist. Otherwise, there are no points of intersection of the graph $q(t)$ with the positive semi-axis.

3. If $C_1<0$, then $q(t)\to -\infty$ as $t\to +\infty$, so there will necessarily be a point of intersection with the $Ot$ axis as $t>0$.

II. Let us now consider the situation when $\lambda_1<0$ (and neither of the constants $C_1$ or $C_2$ is zero). In this case we also have:
$q(t) \sim (1-\frac{C_1}{\lambda_1}-\frac{C_2}{\lambda_2})$ as $t\to +\infty$.

1. If both constants $C_1>0$ and $C_2>0$, then $q'(t)>0$ by virtue of the first equality \eqref{pro1} and the function $q(t)$ increases at $t>0$. There are no points of intersection of its graph with the positive semi-axis.

2. Let $C_1>0$, $C_2<0$. Then the function $q(t)$ necessarily has an extremum point $t_{extr}$, which is found from the equality $q'(t)=0$, whence
$t_{extr}=\frac{1}{\lambda_1-\lambda_2}\ln\left(-\frac{C_2}{C_1}\right)$.
We see that $t_{extr}>0$ if and only if $C_1+C_2<0$. Otherwise, $C_1+C_2 \ge 0$ and
the extremum point does not lie to the right of the $Oq$-axis, since we see from the second equality \eqref{pro1} that $q'(0)$ is non-negative and hence the function $q(t)$ increases for $t>0$. There are no points of intersection of its graph with the positive semi-axis. Let us return to the situation when $C_1+C_2<0$. In this case, the extremum point found is a minimum point, since $q'(0)<0$ by virtue of the second equality \eqref{pro1}. The value of the function at this minimum point  is equal to \eqref{tex} and it must be non-positive for a point of intersection of the graph $q(t)$ with the $Ot$ axis at $t>0$ to exist. Otherwise, there are no points of intersection with $Ot$, $t>0$ axis.

3. Let now $C_1<0$, $C_2>0$. The function $q(t)$ has an extremum point
$t_{extr}=\frac{1}{\lambda_1-\lambda_2}\ln\left(-\frac{C_2}{C_1}\right)$, which is positive if $C_1+C_2>0$. Note that in this case $t_{extr}$ is a maximum point, since $q'(0)>0$ by the second equality \eqref{pro1}. Thus, if we want the graph of $q(t)$ to have an intersection point with the $Ot$ axis at $t>0$, it is necessary and sufficient that the condition
\begin{equation}\label{hh}
\lim\limits
_{t\to +\infty} q(t)=(1-\frac{C_1}{\lambda_1}-\frac{C_2}{\lambda_2})<0,
\end{equation}
will be true, since the function $q(t)$ decreases on the segment from $t_{extr}$ to $+\infty$ .
If $C_1+C_2\le 0$, then the function $q(t)$ decreases when $t>0$ by virtue of the second equality \eqref{pro1}. Thus, the graph of the function $q(t)$ has an intersection point with the $Ot$ axis when $t>0$ if and only if the condition \eqref{hh} is satisfied; otherwise, there is no intersection.

4. Finally, let $C_1<0$, $C_2<0$. In this case, $q'(t)<0$ by virtue of the first equality
\eqref{pro1} and $q(t)$ decreases for all $t>0$. Thus, if $\lim\limits
_{t\to +\infty} q(t)=(1-\frac{C_1}{\lambda_1}-\frac{C_2}{\lambda_2})<0$, then the graph of the function $q(t)$ has an intersection point with the $Ot$ axis for $t>0$, otherwise there is no intersection.

III. Let us now consider the situation when one of the constants $C_1$ or $C_2$ turns out to be zero. We start with the case $C_1=0$, then $q(t)=\frac{C_2}{\lambda_2}(\exp(\lambda_2 t)-1)+1$, $q'(t)=C_2 \exp(\lambda_2 t)$.

1. Let $\lambda_2>0$. If $C_2\ge 0$, then there are no intersection points of the graph $q(t)$ with the positive semi-axis, the function $q(t)$ increases for $t>0$, since the derivative $q'(t)$ is positive. If $C_2<0$, then the function $q(t)\to -\infty$ for $t\to +\infty$, and therefore, there is an intersection point.

2. Let $\lambda_2<0$. If $C_2\ge 0$, then there are no intersection points of the graph of $q(t)$ with the positive semi-axis, the function $q(t)$ increases for $t>0$, since the derivative $q'(t)$ is positive. If $C_2<0$, then the function $q(t)$ decreases for $t>0$, since the derivative $q'(t)$ is negative and $\lim\limits_{t\to +\infty} q(t)=
\left(1-\frac{C_2}{\lambda_2}\right)$. Hence, if $C_2<\lambda_2$, then the intersection point needed exists, otherwise, if $C_2\ge \lambda_2$, it does not exist.

3. We obtain the same conditions from items 1 and 2 for the situation in which $C_2=0$ and $C_1$ is not equal to zero, due to replacing $C_1$ with $C_2$ and $\lambda_1$ with $\lambda_2$. The corollary is proved. In the formulation of the corollary, some of the items considered in this proof were united to obtain a more compact condition for the formation of a singularity. \end{proof}

\begin{corollary}\label{t41}
We suppose $\lambda_1=\lambda_2=\lambda\neq 0$ and define constants $C_1$, $C_2$ by condition \eqref{fr1}. Then the solution of system \eqref{1} with initial data $(v_1(0), v_2(0))$ loses smoothness in finite time if and only if one of the following conditions is satisfied:

\begin{itemize}
\item[1.] $\lambda>0$, $C_1+C_2<0$;

\item[2.] $\lambda<0$, $C_1+C_2<\lambda$.
\end{itemize}
\end{corollary}

\begin{proof} In the case of $\lambda_1=\lambda_2=\lambda\neq 0$ the solution of the system \eqref{1} has the following form: $q(t)=\tfrac{C_1+C_2}{\lambda}(\exp(\lambda t)-1)+1$, and $q'(t)=(C_1+C_2)
\exp(\lambda t)$.

Let $\lambda>0$. If $(C_1+C_2)\ge 0$, then since the derivative $q'(t)$ is non-negative, there is no intersection point of the graph $q(t)$ with the positive semi-axis. If $(C_1+C_2)<0$, then $\lim\limits_{t\to -\infty} q(t)=-\infty$, and there will definitely be an intersection point.

Let $\lambda<0$. If $(C_1+C_2)\ge 0$, then there is no intersection point of the graph $q(t)$ with the positive semiaxis, since the derivative $q'(t)$ is non-negative. If $(C_1+C_2)<0$, then the function $q(t)$ decreases and $\lim\limits_{t\to -\infty} q(t)=1-\frac{C_1+C_2}{\lambda}$. If this limit is negative, then there will be an intersection point, otherwise there is none.

Corollary is proved. \end{proof}

\begin{corollary}\label{t4}
We suppose $\lambda_1>0$, $\lambda_2=0$ and define constants $C_1$, $C_2$ by condition \eqref{fr1}. Then the solution of system \eqref{1} with initial data $(v_1(0), v_2(0))$ loses smoothness in finite time if and only if one of the following conditions is satisfied:

\begin{itemize}
\item[1.] $C_1<0$;

\item[2.] $C_1=0$, $C_2<0$;

\item[3.] $C_1>0$, $C_2<0$, $C_1+C_2<0$, $\lambda_1+C_2\ln\left(-\tfrac{C_2}{C_1}\right) \leq C_1+C_2$.
\end{itemize}
\end{corollary}

\begin{proof}
If $C_1<0$, then $q(t) \sim \tfrac{C_1}{\lambda_1} \exp(\lambda_1 t) \to -\infty$ as $t \to +\infty$, and since $q(0)=1$, there will be an intersection point of the graph of function $q(t)$ with the $Ot$ axis as $t>0$ by virtue of the continuity of $q(t)$.

If $C_1=0$, then $q(t)=1+C_2 t$, and $q(t) \to -\infty$ if and only if
$C_2<0$. Since $q(0)=1$, then there will be a point of intersection with the $Ot$ axis as $t>0$ by virtue of the continuity of $q(t)$. If $C_2\ge 0$, there are no points of intersection.

Let $C_1>0$, then $q'(t)=C_1\exp(\lambda_1 t)+C_2$. Obviously, if $C_2\ge 0$, then the derivative will be positive and there will be no intersection points of the graph $q(t)$ with the $Ot$ axis when $t>0$. If $C_2<0$, then $q(t)$ necessarily has a minimum point, determined by the relation $t_{min}=\tfrac{1}{\lambda_1}\ln\left(-\tfrac{C_2}{C_1}\right)$.
We need to place this point on the right part of $Ot$-axis, i.e. $\tfrac{1}{\lambda_1}\ln\left(-\tfrac{C_2}{C_1}\right)>0$, whence $C_1+C_2<0$. If this condition is not true, then $q'(0)=C_1+C_2 \ge 0$, i.e. for $t>0$ the function $q(t)$ increases and there are no intersection points of its graph with the positive semiaxis. Let us calculate $q(t_{min})=1+\tfrac{C_2}{\lambda_1}\ln\left(-\tfrac{C_2}{C_1}\right)-\tfrac{C_1+C_2}
{\lambda_1}$. To have a point of intersection of the graph $q(t)$ with the $Ot$ axis when $t>0$ (in the case of $C_1+C_2<0$), it is necessary and sufficient to require that $q(t_{min})\le 0$ from which we obtain condition 3 of the corollary.

The corollary is proved.
\end{proof}

\begin{corollary}\label{t5}
We suppose $\lambda_1=0$, $\lambda_2<0$ and define $C_1$, $C_2$ by condition \eqref{fr1}. Then the solution of system \eqref{1} with initial data $(v_1(0), v_2(0))$ loses smoothness in finite time if and only if one of the following conditions is satisfied:

\begin{itemize}
\item[1.] $C_1<0$;

\item[2.] $C_1=0$, $C_2<\lambda_2$;

\item[3.] $C_1>0$, $C_2<0$, $C_1+C_2<0$, $\lambda_2+C_1\ln\left(-\tfrac{C_1}{C_2}\right) \geq C_1+C_2$.
\end{itemize}
\end{corollary}

\begin{proof} If $C_1<0$, then $q(t) \sim
C_1 t \to -\infty$ as $t \to +\infty$. Since $q(0)=1$, there will necessarily be a point of intersection of its graph with the $Ot$ axis as $t>0$ due to the continuity of the function $q(t)$.

If $C_1=0$, then $q(t)=1+\tfrac{C_2}{\lambda_2}(\exp(\lambda_2 t)-1)$ and $q'(t)
=C_2\exp(\lambda_2 t)$. If $C_2>0$, then $q'(t)$ is positive and $q(t)$ increases, which means there are no points of intersection of its graph with the $Ot$ axis when $t>0$. If $C_2=0$, then $q(t)\equiv 1$, and there are no intersection points with the positive semi-axis either. If $C_2<0$, then the function $q(t)$ decreases for all $t$ and tends to the value $q_{inf}=
1-\frac{C_2}{\lambda_2}$ as $t \to +\infty$, therefore there will be a point of intersection of the graph with the $Ot$ axis as $t>0$ if and only if this value $q_{inf}$ is negative. We obtain the condition from point 2 of the theorem from this demonstration.

Let $C_1>0$, then $q'(t)=C_1+C_2\exp(\lambda_2 t)$. Obviously, if $C_2\ge 0$, then the derivative will be positive and there will be no intersection points of the graph with $Ot$ axis at $t>0$. If $C_2<0$, then $q(t)$ necessarily has a minimum point which satisfies the relation $t_{min}=\tfrac{1}{\lambda_2}\ln\left(-\tfrac{C_1}{C_2}\right)$.
We need to place this point on the right part of $Ot$-axis, i.e. $\tfrac{1}{\lambda_2}\ln\left(-\tfrac{C_1}{C_2}\right)>0$, whence $C_1+C_2<0$. If this condition is not satisfied, then $q'(0)=C_1+C_2 \ge 0$, i.e. for all $t>0$ the function $q(t)$ increases and there are no intersection points with the positive semi-axis. Let us calculate $q(t_{min})=1+\tfrac{C_1}{\lambda_2}\ln\left(-\tfrac{C_1}{C_2}\right)-\tfrac{C_1+C_2}
{\lambda_2}$. To have a point of intersection of the graph with the $Ot$ axis for $t>0$ (in the case of $C_1+C_2>0$), it is necessary and sufficient to require that $q(t_{min})\le 0$, whence we obtain condition 3 of the corollary.

The corollary is proved. \end{proof}

\subsection{Case ${\mathbb B}$ ($J_{\mathbb B}$ is a Jordan cell)} We define constants $C_1$, $C_2$ in terms of initial data $v_1(0)$ and $v_2(0)$ as follows:
\begin{equation}\label{fr2}
C_1=\frac{a_{22}}{\det A}(a_{11}v_1(0)+a_{12}v_2(0)),\quad C_2=\frac{a_{22}-a_{12}}{\det A}(a_{21}v_1(0)+a_{22}v_2(0)).
\end{equation}

\noindent 
To find the function $q(t)=\int u_1(t)\,dt=\int (C_1 +C_2 t) \exp(\lambda t)\,dt$, $q(0)=1$,
we need to consider two following situations:

\begin{itemize}
\item[1.] if $\lambda \neq 0$, then
$$q(t)=1+\frac{C_2}{\lambda}t\exp(\lambda t)+\frac{(C_1 \lambda - C_2)(
\exp(\lambda t)-1)}{\lambda^2};$$

\item[2.] if $\lambda = 0$, then
$$q(t)=1 + C_1 t+ \frac{C_2 t^2}{2}.$$
\end{itemize}

We will now prove two corollaries of Theorem \ref{t2} for two cases mentioned above. Corollary \ref{t61} considers the case of nonzero eigenvalues; Corollary \ref{t6} considers the case of coinciding zero eigenvalues.

\begin{corollary}\label{t61}
Let $\lambda\neq 0$, constants $C_1$, $C_2$ are defined by condition \eqref{fr2}. Then the solution of system \eqref{1} with initial data $(v_1(0), v_2(0))$ loses smoothness in finite time if and only if one of the following situations occurs:

\begin{itemize}
\item[1.] when $\lambda>0$

\begin{itemize}
\item[a.] $C_2<0$;

\item[b.] $C_1<0$, $C_2>0$, $C_2 (\exp(-\tfrac{C_1 \lambda}
{C_2})-1)+C_1 \lambda-\lambda^2 \ge 0$;

\item[c.] $C_2=0$, $C_1<0$;
\end{itemize}

\item[2.] when $\lambda<0$

\begin{itemize}
\item[a.] $C_2< 0$, $1-C_1 \lambda^{-1}+C_2 \lambda^{-2}<0$;

\item[b.] $C_1<0$, $C_2>0$, $C_2 (\exp(-\tfrac{C_1 \lambda}
{C_2})-1)+C_1 \lambda-\lambda^2 \ge 0$.
\end{itemize}
\end{itemize}

\end{corollary}

\begin{proof} Let us consider the cases one by one. We note firstly that
\begin{equation}\label{pro2}
q'(t)=(C_1+C_2 t)e^{\lambda t},\quad q'(0)=C_1.
\end{equation}

I. Let us suppose that $\lambda>0$, then $q(t) \sim C_2 t \tfrac{e^{\lambda t}}{\lambda}$ as $t \to +\infty$.

1. If $C_2<0$ then $\lim\limits_{t \to +\infty} q(t)=-\infty$, and since $q(0)=1$, there is always a point of intersection of the graph of the function $q(t)$ with $Ot$ axis as $t>0$.

2. If $C_2>0$, then the function $q(t)$ has an extremum point $t_{extr}=-\tfrac{C_1}{C_2}$. It is easy to see that $t_{extr}>0$ if and only if $C_1 < 0$. Indeed, otherwise, when $C_1\ge 0$, we have by virtue of the second equality \eqref{pro2} that the function $q(t)$ increases for $t>0$, so there are no points of intersection of its graph with $Ot$ axis for $t>0$. If $C_1<0$, then the point $t_{extr}$ will be a minimum point (since $q'(0)<0$). Thus, on the positive semiaxis there is a point of intersection of the graph of the function $q(t)$ if and only if
\begin{equation}\label{tex1}
q(t_{extr})=-C_2 \lambda^{-2}(\exp(-\tfrac{C_1 \lambda}
{C_2})-1)-C_1 \lambda^{-1}+1 \le 0,
\end{equation}
and it won't exist otherwise.

II. Let us now suppose that $\lambda<0$ then $\lim\limits_{t \to +\infty} q(t)=1-C_1 \lambda^{-1} +C_2 \lambda^{-2}$.

1.If $C_1>0$ and $C_2>0$, then the derivative $q'(t)$ is positive by virtue of the first equality \eqref{pro2}, and the function $q(t)$ increases when $t>0$, so there are no points of intersection of its graph with the positive semi-axis.

2. If $C_1>0$ and $C_2<0$, then the function $q(t)$ has an extremum point defined by the equality $t_{extr}=-\frac{C_1}{C_2}>0$ and it is a maximum point since $q'(0)=C_1>0$. Thus, there is a point of intersection of the graph with the $Ot$ axis at $t>0$ if and only if $\lim\limits_{t \to +\infty} q(t)=1-C_1 \lambda^{-1} +C_2 \lambda^{-2}<0$, and there is no such point of intersection otherwise.

3. If $C_1<0$ and $C_2>0$, then the function $q(t)$ has an extremum point defined by the equality $t_{extr}=-\frac{C_1}{C_2}>0$ and it is a minimum point, since $q'(0)=C_1<0$. This means that there will be a point of intersection with $Ot$ axis as $t>0$ if and only if the inequality \eqref{tex1} holds, and there is none otherwise.

4. If $C_1<0$ and $C_2<0$, then the derivative $q'(t)$ is negative by virtue of the first equality \eqref{pro2}, and the function $q(t)$ decreases for $t>0$. Thus, there is a point of intersection of its graph with $Ot$ axis for $t>0$ if and only if $\lim\limits_{t \to +\infty} q(t)=1-C_1 \lambda^{-1} +C_2 \lambda^{-2}<0$, and there is no point of intersection otherwise.

III. Let us consider a situation, when $C_1=0$ then $q(t)=-C_2 \lambda^{-2} (e^{\lambda t}-1)+C_2 \lambda^{-1} t e^{\lambda t} +1$, $q'(t)=C_2 t e^{\lambda t}$.

1. Let $C_2>0$, then the function $q(t)$ increases when $t>0$ by virtue of the first equality \eqref{pro2}, therefore, there is no point of intersection with the positive semi-axis.

2. If $C_2<0$ and $\lambda>0$, then $q(t)\sim C_2 \lambda^{-2} t e^{\lambda t} \to
-\infty$ as $t\to +\infty$. This means that there will necessarily be a point of intersection with $Ot$ axis as $t>0$. If $\lambda<0$, then the function $q(t)$ decreases as $t>0$ by virtue of the first equality \eqref{pro2}, and there is a point of intersection with $Ot$ axis, $t>0$ if and only if $\lim\limits_{t \to +\infty} q(t)=(1+C_2 \lambda^{-2})<0$, i.e.
$C_2< -\lambda^2$. Otherwise, there is no point of intersection with the positive semi-axis.

IV. Let $C_2=0$ then $q(t)=C_1 \lambda^{-1} (e^{\lambda t}-1) +1$, $q'(t)=C_1 e^{\lambda t}$.

1. Let $C_1>0$, then the function $q(t)$ increases when $t>0$ by virtue of the first equality \eqref{pro2}, therefore, there is no intersection point of its graph with the positive semi-axis.

2. If $C_1<0$ and $\lambda>0$, then $q(t)\sim C_1 \lambda^{-1} e^{\lambda t} \to
-\infty$ as $t\to +\infty$. This means that there will necessarily be a point of intersection with $Ot$ axis as $t>0$. If $\lambda<0$, then the function $q(t)$ decreases as $t>0$ by virtue of the first equality \eqref{pro2}, and there is a point of intersection with the positive semi-axis if and only if $\lim\limits_{t \to +\infty} q(t)= (1-C_1 \lambda^{-1})<0$, i.e.
$C_1< \lambda$. Otherwise, there is no point of intersection with $Ot$ axis as $t>0$.

Corollary is proved. \end{proof}

\begin{corollary}\label{t6}
Let us define $\lambda=0$ and constants $C_1$, $C_2$ by condition \eqref{fr2}. Then the solution of system \eqref{1} with initial data $(v_1(0), v_2(0))$ loses smoothness in finite time if and only if one of the following conditions is satisfied:

\begin{itemize}
\item[1.] $C_2<0$;

\item[2.] $C_2=0$, $C_1<0$;

\item[3.] $C_2>0$, $C_1<0$, $C_1^2\ge 2 C_2$.
\end{itemize}
\end{corollary}

\begin {proof} If $C_2<0$, then parabola $q(t)$ opens downwards and $q(0)=1>0$, therefore, there is a point of intersection with $Ot$ axis at $t>0$.

If $C_2=0$, then $q(t)=1+C_1 t$. So, if $C_1\ge 0$, we have $q(t)$ which increases, and
there are no points of intersection with $Ot$ axis at $t>0$, if $C_1 <0$, then $q(t) \to -\infty$
at $t\to +\infty$ and the intersection point needed is found.

If $C_2 >0$, then the parabola opens upward, its vertex is located at the point $t_{min}=-\frac{C_1}{C_2}$. If $C_1 \ge 0$, then the parabola vertex lies at the left half-plane or on the $Oq$ axis, and on the interval $t\ge 0$ the function $q(t)$ increases, so there are no points of intersection of its graph with the positive half-axis. If $C_1<0$, then the vertex lies at the right half-plane and $q(t_{min})=1-\frac{C_1^2}{2C_2} \le 0$, from which we obtain the third condition of the corollary.

Corollary is proved. \end{proof}

\subsection{Case ${\mathbb C}$ ($J_{\mathbb C}$ is a skew-symmetric matrix)} Let us define constants $C_1$, $C_2$ in terms of initial data $v_1(0)$ and $v_2(0)$ as follows:
\begin{equation}\label{fr3}
C_1=\frac{(a_{11}a_{12}-a_{21}a_{22})v_1(0)+(a_{12}^2+a_{22}^2)v_2(0)}{\det A},\quad C_2=v_1(0).
\end{equation}

\noindent
Then we have $q(t)=\int u_1(t)\,dt=\int \exp(\alpha t)$ $(C_1 \sin\beta t +C_2 \cos\beta t)\,dt$, $q(0)=1$,
and we get by direct integration

$$q(t)=\frac{e^{\alpha t}}{\alpha^2+\beta^2}[(C_1 \alpha+C_2 \beta)\sin\beta t
+(C_2 \alpha - C_1 \beta)\cos \beta t]-\frac{C_2 \alpha
-C_1 \beta}{\alpha^2+\beta^2}+1.$$

In this case we have another corollary from theorem \ref{t2}:

\begin{corollary}\label{t7}
We suppose $\lambda_1=\overline{\lambda_2}\in\mathbb{C}$ and define constants $C_1$, $C_2$ by condition \eqref{fr3}. Let also constants $C_1$ and $C_2$ not equal zero simultaneously. Then the solution of system \eqref{1} with initial data $(v_1(0), v_2(0))$ loses smoothness in finite time if and only if one of the following situations occurs:

\begin{itemize}
\item[1.] $\alpha>0$;

\item[2.] $\alpha=0$, $\beta^2+2C_1 \beta \le C_2^2$;

\item[3.] when $\alpha<0$

\begin{itemize}
\item[a.] $C_2\ge 0$, $C_1\ge 0$, $-\beta e^{\alpha t_2^*} \sqrt{C_1^2 +C_2^2}
\le C_2 \alpha - C_1 \beta - \alpha^2 - \beta^2
$;

\item[b.] $C_2 \ge 0$, $C_1 \le 0$, $e^{\alpha t_1^*} \left(\beta \sqrt{C_1^2 +C_2^2}
+\frac{2C_1 C_2 \alpha}{\sqrt{C_1^2 +C_2^2}}\right)
\ge C_1 \beta + \alpha^2 + \beta^2 - C_2 \alpha
$;

\item[c.] $C_2 \le 0$, $C_1 \ge 0$, $e^{\alpha t_0^*} \left(\beta \sqrt{C_1^2 +C_2^2}
+\frac{2C_1 C_2 \alpha}{\sqrt{C_1^2 +C_2^2}}\right)
\le C_2 \alpha - C_1 \beta - \alpha^2 - \beta^2
$;

\item[d.] $C_2\le 0$, $C_1\le 0$, $\beta e^{\alpha t_1^*} \sqrt{C_1^2 +C_2^2}
\le C_2 \alpha - C_1 \beta - \alpha^2 - \beta^2
$.
\end{itemize}
\end{itemize}

\noindent Here $t^{*}_n=\frac{1}{\beta}(\pi n -\arctan\tfrac{C_2}{C_1})$, $n\in\,\mathbb{Z}$.

\end{corollary}

\begin{proof} Let us consider the cases one by one. We note that
\begin{equation}\label{pro3}
q'(t)=e^{\alpha t}(C_1 \sin{\beta t} +C_2 \cos{\beta t}),\quad q'(0)=C_2.
\end{equation}
Let us transform the function $q(t)$ as follows:
\begin{equation}\label{trans}
q(t)=\frac{\exp(\alpha t)}{\sqrt{\alpha^2 + \beta^2}}\sqrt{C_1^2 + C_2^2}\sin(\beta t+ \phi_0)-\frac{1}{\alpha^2+\beta^2}(C_2 \alpha
-C_1 \beta)+1,
\end{equation}
where $\phi_0$ is an auxiliary angle such that $\sin\phi_0=
\frac{C_2 \alpha - C_1 \beta}{\sqrt{(C_1^2+C_2^2)(\alpha^2+\beta^2)}}$; $\cos\phi_0=
\frac{C_1 \alpha + C_2 \beta}{\sqrt{(C_1^2+C_2^2)(\alpha^2+\beta^2)}}$.

1. Let $\alpha>0$. Then it follows from the formula \eqref{trans} that
$\limsup\limits_{t\to +\infty} q(t) = +\infty$ and $\liminf\limits_{t\to +\infty} q(t) = -\infty$, and we obtain that there are surely points of intersection of the function $q(t)$ with $Ot$ axis for $t>0$.

2. Let $\alpha=0$, in this case $q(t)$ has the following form:
$$q(t)=\frac{1}{\beta}(C_2 \sin{\beta t} - C_1\cos{\beta t})+\frac{C_1}{\beta}+1.$$
Using equality \eqref{trans} we get
$$q_{min}=-\frac{1}{\beta}\sqrt{C_1^2+C_2^2}+1+\frac{C_1}{\beta}
\le q(t) \le -\frac{1}{\beta}\sqrt{C_1^2+C_2^2}+1+\frac{C_1}{\beta}=q_{max},$$
and $q(t)$ assumes all values on the interval $[q_{min},\,q_{max}]$. Thus, the graph $q(t)$ intersects $Ot$-axis at $t>0$ if and only if $q_{min}\le 0\le q_{max}$, whence $(\tfrac{C_1}{\beta}+1)^2 \le \tfrac{1}{\beta^2}(C_1^2+C_2^2)$. Otherwise, there are no intersection points with $Ot$-axis at $t>0$.

3. Let $\alpha<0$. In this case, the function $q(t)$ has a countable number of extrema $t_n^*$, which are found from the equation $q'(t)=0$, i.e. $\tan{\beta t}=-\frac{C_2}{C_1}$, whence $\beta t_n^* = \pi n - \arctan\frac{C_2}{C_1}$, $n\in\mathbb{Z}$. We note that in this case $q(t)$ intersects $Ot$ axis as $t>0$ if and only if $q=0$ lies above the value $q(t_{min})$ at the first minimum point of the function $q(t)$ on the positive semiaxis (the maximum value of the function $q(t)$ at $t>0$ is positive, since $q(0)=1$). In fact, since the amplitude of oscillations of the function $q(t)$ decreases when $\alpha<0$, the values at the following minimum points will rise, and the values at the following maximum points will decrease.

a. Let $C_1 \ge 0$, $C_2 \ge 0$. In this case, the first extremum point at $t>0$ will be $t_1^*=\tfrac{1}{\beta}(\pi-\arctan\tfrac{C_2}{C_1})$, and the next one $t_2^*=\tfrac{1}{\beta}(2\pi-\arctan\tfrac{C_2}{C_1})$. Then $q'(0)>0$ by virtue of the equality \eqref{pro3} and $t_1^*$ is the maximum point, $t_2^*$ is the minimum point. Obviously,
the point $\beta t^*_2$ lies in the fourth quadrant, then $\sin{\beta t^*_2}=-\tfrac{C_2}{\sqrt{C_1^2 +C_2^2}}$, $\cos{\beta t^*_2}=\tfrac{C_1}{\sqrt{C_1^2 +C_2^2}}$. Direct calculations show that
$$q(t^*_2)=-\tfrac{\beta e^{\alpha t_2^*}}{\alpha^2+\beta^2} \sqrt{C_1^2 +C_2^2}
-\tfrac{1}{\alpha^2+\beta^2}(C_2 \alpha - C_1 \beta)+1,$$
from which we obtain the inequality indicated at 3a item of the corollary.

b. Let $C_1 \le 0$, $C_2 \ge 0$. In this case, the first extremum point at $t>0$ will be $t_0^*=-\tfrac{1}{\beta}\arctan\tfrac{C_2}{C_1}$, and the next one $t_1^*=\tfrac{1}{\beta}(\pi-\arctan\tfrac{C_2}{C_1})$. Then $q'(0)>0$  by virtue of the equality \eqref{pro3}, and hence $t_0^*$ is a maximum point, $t_1^*$ is a minimum point. Obviously,
the point $\beta t^*_1$ lies in the third quadrant, then $\sin{\beta t^*_1}=-\tfrac{C_2}{\sqrt{C_1^2 +C_2^2}}$, $\cos{\beta t^*_1}=-\tfrac{C_1}{\sqrt{C_1^2 +C_2^2}}$. Direct calculations show that
$$q(t^*_1)=-\tfrac{e^{\alpha t_1^*}}{\alpha^2+\beta^2} \left(\beta \sqrt{C_1^2 +C_2^2}+\frac{2C_1 C_2 \alpha}{\sqrt{C_1^2 +C_2^2}}\right)
-\tfrac{1}{\alpha^2+\beta^2}(C_2 \alpha - C_1 \beta)+1,$$
from which we obtain the inequality indicated at 3b item of the corollary.

c. Let $C_1 \ge 0$, $C_2 \le 0$. In this case, the first extremum point at $t>0$ will be $t_0^*=-\tfrac{1}{\beta}\arctan\tfrac{C_2}{C_1}$, and the next one $t_1^*=\tfrac{1}{\beta}(\pi-\arctan\tfrac{C_2}{C_1})$. Then $q'(0)<0$ by virtue of the equality \eqref{pro3}, and hence $t_0^*$ is a minimum point, $t_1^*$ is a maximum point. Obviously, the point $\beta t^*_0$ lies in the first quadrant, then $\sin{\beta t^*_0}=\tfrac{C_2}{\sqrt{C_1^2 +C_2^2}}$, $\cos{\beta t^*_0}=\tfrac{C_1}{\sqrt{C_1^2 +C_2^2}}$.
Since
$$q(t^*_0)=\tfrac{e^{\alpha t_0^*}}{\alpha^2+\beta^2} \left(\beta \sqrt{C_1^2 +C_2^2}+\frac{2C_1 C_2 \alpha}{\sqrt{C_1^2 +C_2^2}}\right)
-\tfrac{1}{\alpha^2+\beta^2}(C_2 \alpha - C_1 \beta)+1,$$
we obtain the inequality indicated at 3c item of the corollary.

d. Let $C_1 \le 0$, $C_2 \le 0$. In this case, the first extremum point at $t>0$ will be $t_1^*=\tfrac{1}{\beta}(\pi-\arctan\tfrac{C_2}{C_1})$, and the next one $t_2^*=\tfrac{1}{\beta}(2\pi-\arctan\tfrac{C_2}{C_1})$. Then $q'(0)<0$ by virtue of the equality \eqref{pro3}, and hence $t_1^*$ is a minimum point, $t_2^*$ is a maximum point. Obviously, the point $\beta t^*_1$ lies in the second quadrant, then $\sin{\beta t^*_1}=\tfrac{C_2}{\sqrt{C_1^2 +C_2^2}}$, $\cos{\beta t^*_1}=-\tfrac{C_1}{\sqrt{C_1^2 +C_2^2}}$.
Calculating
$$q(t^*_1)=\tfrac{\beta e^{\alpha t_1^*}}{\alpha^2+\beta^2} \sqrt{C_1^2 +C_2^2}
-\tfrac{1}{\alpha^2+\beta^2}(C_2 \alpha - C_1 \beta)+1,$$
we obtain the inequality indicated at 3d item of the corollary.

Corollary is proved. \end{proof}

\section{Singular points of the extended system \eqref{1} for derivatives}

The system \eqref{1} has equilibria indicated in the table. Their stability / instability was studied using Lyapunov's theorem on stability and is given below:

\noindent
\begin{center}
\begin{tabular}{ | m{3 cm} | m{4 cm} | m{4 cm} | }
\hline
equilibrium & asympt. stable & unstable\\
\hline
$B_1(0,0)$ & $\Re(\lambda_1)<0$ & $\Re(\lambda_1)>0$ and (or) $\Re(\lambda_2)>0$\\
\hline
$B_2(\lambda_2,\,-\tfrac{a_{11}}{a_{12}}\lambda_2)$ & \phantom{} & $\lambda_1 > \lambda_2$ and (or) $\lambda_2<0$\\
\hline
$B_3(\lambda_1,\,-\tfrac{a_{21}}{a_{22}}\lambda_1)$ & $\lambda_1>\lambda_2$ and $\lambda_1>0$ & $\lambda_1<0$\\
\hline
$B_4(\lambda,\,-\tfrac{a_{21}}{a_{22}}\lambda)$ & \phantom{} & $\lambda<0$\\
\hline
\end{tabular}
\end{center}

The equilbria $B_2$ ($a_{12}\neq 0$) and $B_3$ ($a_{22}\neq 0$) exist only in the case of the matrix $J_{\mathbb A}$. The equilibrium $B_4$ ($a_{22}\neq 0$) exists for the case of the matrix $J_{\mathbb B}$.


The table below classifies these singular points on the phase plane:

\noindent
\begin{tabular}{ | m{1.4 cm} | m{2.6 cm} | m{2.6 cm} | m{1.5 cm} | m{1.5 cm} | m{1.2 cm}| }
 \hline
 point & node & saddle & dicrit. node & degener. node & focus \\
 \hline
 $B_1$ ($\lambda_1$ $\neq$ $\lambda_2$ $\in$ $\mathbb{R}$) & $\lambda_1 \lambda_2>0$, $\lambda_1 \neq \lambda_2$
 & $\lambda_1 \lambda_2<0$, $\lambda_1 \neq \lambda_2$ & $\lambda_1$ $=$ $\lambda_2$ $\neq 0$ & \phantom{} & \phantom{}\\
 \hline
 $B_1$  ($\lambda_1$ $=$ $\lambda_2$) & \phantom{} & \phantom{} & \phantom{} & $\lambda \neq 0$ & \phantom{}\\
 \hline
 $B_1$  ($\lambda_1,$ $\lambda_2\in\mathbb{C}$) & \phantom{} & \phantom{} & \phantom{} &  \phantom{} &
 $\alpha\neq 0$\\
 \hline
 $B_2$ & $(\lambda_2-\lambda_1)\lambda_2>0$, $\lambda_1\neq 0$ & $(\lambda_2-\lambda_1)\lambda_2<0$, $\lambda_1\neq 0$ & $\lambda_1=0$, $\lambda_2 \neq 0$, $a_{22}=0$ &  $\lambda_1=0$, $\lambda_2 \neq 0$, $a_{22}\neq 0$ & \phantom{}\\
 \hline
 $B_3$ & $(\lambda_1-\lambda_2)\lambda_1>0$, $\lambda_2\neq 0$ & $(\lambda_1-\lambda_2)\lambda_1<0$, $\lambda_2\neq 0$ & $\lambda_1\neq 0$, $\lambda_2 = 0$, $a_{12}=0$ &  $\lambda_1\neq 0$, $\lambda_2 = 0$, $a_{12}\neq 0$ & \phantom{}\\
 \hline

\end{tabular}

\begin{proof} We find first the equilibria of the system \eqref{ss} (on the phase plane ${\bf w}$), then, using the equality \eqref{vec}, the equilibria of the system \eqref{1} (on the phase plane ${\bf v}$) and study them by the first approximation. We denote by $\mu_1$, $\mu_2$ the eigenvalues of the matrix of the linear approximation of the right-hand side of the system \eqref{ss}.

1. For the case of the matrix $J_{\mathbb A}$, there are three equilibria of the system \eqref{1}: $\tilde B_1(0,\,0)$, $\tilde B_2 \left(0,\,-\tfrac{\lambda_2 \det A}{a_{12}}\right)$, $\tilde B_3 \left(\tfrac{\lambda_1 \det A}{a_{22}},\,0\right)$, which correspond in turn to three points $B_i$, $i=1..3$. For the equilibrium $B_1$, the eigenvalues are: $\mu_1=\lambda_1$, $\mu_2=\lambda_2$; for the equilibrium $B_2$ the eigenvalues equal $\mu_1=\lambda_1-\lambda_2$, $\mu_2=-\lambda_2$, for the equilibrium $B_3$ we have the eigenvalues $\mu_1=-\lambda_1$, $\mu_2=\lambda_2-\lambda_1$. Applying Lyapunov's theorem on the stability by the first approximation, we obtain the conditions indicated in the table.

2. For the case of the matrix $J_{\mathbb B}$, there are two equlibria of the system \eqref{1}: $\tilde B_1$ and $\tilde B_4 \left(\tfrac{\lambda \det A}{a_{22}},\,0\right)$, which correspond in turn to two points $B_1$ and $B_4$. For the equilibrium $B_1$, the eigenvalues are $\mu_1=\mu_2=\lambda$; for the equilibrium $B_4$ the eigenvalues equal $\mu_1=-\lambda$, $\mu_2=0$. Applying Lyapunov's theorem on the stability, we obtain the conditions indicated in the table.

3. For the case of matrix $J_{\mathbb C}$ there is one equilibrium of the system \eqref{1}: $\tilde B_1$, which corresponds to the point $B_1$. Its eigenvalues are $\mu_1=\alpha+i\beta,\,\mu_2=\alpha-i\beta$. Applying Lyapunov's theorem on the stability, we obtain the conditions indicated in the table. The statement is proved.
\end{proof}

Note that the corollaries of the theorem \ref{t2} allow us to find precisely domains of attraction for singular points. The table below shows the relations on $\lambda_i$, $i=1,2$, for which the study of a given point is impossible according to Lyapunov's theorem:

\noindent
\begin{center}
\begin{tabular}{ | m{1 cm} | m{4 cm}| m{5 cm}|}
\hline
$B_1$ & $\Re(\lambda_1)=0$ & unstable for $\lambda_1\in\mathbb{R}$, non-asympt. stable for $\lambda_1\in\mathbb{C}$\\
\hline
$B_2$ & $\lambda_1=\lambda_2$ & unstable\\
\hline
$B_3$ & $\lambda_1=\lambda_2$ or $\lambda_1=0$ & unstable\\
\hline
$B_4$ & $\lambda \ge 0$ & unstable\\
\hline
\end{tabular}
\end{center}

\begin{proof} To prove the instability of the singular point $B_1$ in the cases $J_{\mathbb A}$ or $J_{\mathbb B}$, we construct the Chetaev function of the form $V(w_1,w_2)=w_1^2-w_2^2$ in the region $-w_1<w_2<w_1$ or $V(w_1,w_2)=w_1 w_2$ in the region $w_1>0$, $w_2>0$, respectively.

To prove the stability of the point $B_1$ in the case of $J_{\mathbb C}$, we construct a Lyapunov function. It is the first integral of the system \eqref{ss}, which has the following form:
$$V(w_1,\,w_2)=\left(\tfrac{\beta\det{A}\cdot (a_{12} \sqrt{w_1^2+w_2^2}-|a_{12}w_1
+a_{22} w_2 +\beta \det{A}|)}{a_{12}^2 w_2^2-2a_{12}a_{22}w_1 w_2 - 2\beta \det{A}\cdot a_{12}w_{1}-(a_{22}w_2
+\beta\det{A})^2}\right)^2 (w_1^2+w_2^2).$$

To prove the instability of singular points $B_2-B_4$, we construct a Chetaev function of the form $V(w_1,w_2)=w_1 w_2$ in the region $w_1>0$, $w_2>0$, or $w_1<0$, $w_2<0$, shifted to each of these points, respectively. \end{proof}

\section{Simple waves}
The solution of a system of two quasilinear equations for an unknown vector $(V_1, V_2)^T$ is called a simple wave if the components of the solution $V_1$ and $V_2$ are related by the dependence $V_2=\Psi(V_1)$, where $\Psi$ is arbitrary, at least $C^1$-smooth function. This allows us to reduce the original system of simple waves to a single quasilinear equation, which can then be studied by standard methods (e.g., \cite{Kruzhkov}).

For an arbitrary system, a simple wave may not exist. We will show that for the system \eqref{gg} a simple wave always exists, although the connection between the functions $V_1$ and $V_2$ may not be explicit. However, if the dependence $V_2=\Psi(V_1)$, $\Psi'(V_1)=0$, is found, the equation to which \eqref{ggg} is reduced has the form
$$\frac{\partial V_1}{\partial t}+V_1 \frac{\partial V_1}{\partial x}=a V_1 + b \Psi(V_1).$$

Indeed, if we assume the dependence $V_2=V_2(V_1)$, then along each characteristic $x=x(t)$ the system \eqref{ggg} can be reduced to one homogeneous equation
\begin{equation}\label{y}
\frac{dV_2}{dV_1}=\frac{cV_1 +dV_2}{aV_1+bV_2},
\end{equation}
which first integral can be found analytically and gives the necessary relationship between $V_1$ and $V_2$.

We will find out what form this first integral has depending on the form of the matrix $Q$. Let us make a change of variables
\begin{equation}\label{syst}
{\bf W}=A{\bf V},\quad {\bf W}=(W_1\,\, W_2)^{T}, \quad {\bf V}=(V_1\,\, V_2)^{T},
\end{equation} where $A$ is a non-singular transition matrix from $Q$ to its Jordan normal form $J$ such that $J=AQA^{-1}$, then we can write the system \eqref{sy} as $\dot {\bf W}=J {\bf W}$. If the matrix has the form $A=\left(\begin{smallmatrix} a_{11} & a_{12}\\ a_{21} & a_{22}\end{smallmatrix}\right)$, then
$$W_1=a_{11} V_1 + a_{12} V_2,\quad W_2=a_{21} V_1 + a_{22} V_2.$$

\begin{theorem}
Let us denote by $\lambda_{1,2}$ the eigenvalues of the matrix $A$ which satisfy the equation $\lambda^2-(a+d)\lambda+(ad-bc)=0$.

\begin{itemize}
\item[1.] Let $\lambda_1$ and $\lambda_2$ be real, and each of them has the same number of independent eigenvectors as the multiplicity of this eigenvalue. Then the first integral of the equation \eqref{y} has the form: $(a_{21} V_1 + a_{22} V_2)^{\lambda_1}$ $=C (a_{11} V_1 + a_{12} V_2)^{\lambda_2}$.

\item[2.] If $\lambda_1$ and $\lambda_2$ are real numbers, which coincide, and there exists only one eigenvector, then the first integral of the equation \eqref{y} has the form: $\lambda (a_{11} V_1 + a_{12} V_2) = (a_{21} V_1 + a_{22} V_2)(C+\ln(a_{21} V_1 + a_{22} V_2))$, where $\lambda=\lambda_i$, $i=1,2$.

\item[3.] If $\lambda_1$ and $\lambda_2$ are complex conjugate eigenvalues, then the first integral of equation \eqref{y} has the form: $\beta\ln((a_{11} V_1 + a_{12} V_2)^2+(a_{21} V_1 + a_{22} V_2)^2) +2\alpha\arctan\frac{a_{21} V_1 + a_{22} V_2}{a_{11} V_1 + a_{12} V_2}=C,$ where $\alpha = \Re(\lambda_1)$, $\beta=\Im(\lambda_1)$ and $\Im(\lambda_1)>0$.
\end{itemize}

Here in the items 1-3 of the theorem $C$ is an arbitrary constant.
\end{theorem}

\begin{proof}

1. For the case of the matrix $J_{\mathbb A}$ after the transition to new unknown functions \eqref{syst} the equation \eqref{y} has the form: $\frac{dW_2}{dW_1}
=\frac{\lambda_2 W_2}{\lambda_1 W_1}$, whence $W_2^{\lambda_1}=C W_1^{\lambda_2}$.

2. For the case of the matrix $J_{\mathbb B}$ after the transition to new unknown functions \eqref{syst} the equation \eqref{y} has the form:  $\frac{dW_2}{dW_1}=\frac{\lambda W_2}{\lambda W_1+W_2}$, whence $\lambda W_1 = W_2 (C+\ln W_2)$.

3. For the case of the matrix $J_{\mathbb C}$ after the transition to new unknown functions \eqref{syst} we get: $\frac{dW_2}{dW_1}=\frac{-\beta W_1+\alpha W_2}{\alpha W_1+\beta W_2}$, whence $\beta\ln(W_1^2+W_2^2) +2\alpha\arctan\frac{W_2}{W_1}= C$. The theorem is proved. \end{proof}

\begin{proposition}
The integral curves of the equation \eqref{y} are bounded if and only if the following two conditions are satisfied: $a=-d$; $d^2+bc<0$.
\end{proposition}

\begin{proof} The result follows from the fact that the integral trajectories of the linear equation \eqref{y} on the plane $(V_1,\,V_2)$ are bounded if and only if the singular point $V_1=V_2=0$ is the center.
 \end{proof}

Simple and traveling waves for the Euler-Poisson equations \eqref{ep1} in the case $k=-1$, $N=1$, $\gamma=0$ corresponding to this case are found in \cite{Ros2}. Let us find simple waves of these equations in the remaining four model cases with $\gamma=0$.

The system of equations \eqref{sy} on the characteristics has the form
\begin{equation}\label{char}
\tfrac{dV}{dt}=kE,\quad \tfrac{dE}{dt}=N V,
\end{equation}
and we obtain from it the first integral of the form
\begin{equation}\label{1i}
N V^2 - k E^2 = C.
\end{equation}
In this case, it is possible to express one component of the solution through the other
and reduce the system \eqref{ep1} to one equation
$$V_t + V V_x=kE(V)-q V,\quad E(V)=\pm \sqrt{\frac{1}{k}(N V^2 - C)}.$$

Let us denote $V_x=v$, $E_x=e$. From \eqref{1i} it follows that $e = \frac{N V v}{k E(V)}$, and thus the first equation of the system \eqref{1} takes the form
\begin{equation}\label{ww}
\dot v = -v^2 + \frac{N V v}{E(V)}
\end{equation}
Thus, considering the equation \eqref{ww} together with the first equation \eqref{char} we obtain a linear equation for the variable $s=v^{-1}$:
$$\frac{ds}{dV}+\frac{N V s}{N V^2 - C}=\pm \frac{{\rm \mbox  sign} (k)}{\sqrt{k(N V^2 - C)}}.$$
After integration we get
$$v=\pm \frac{\sqrt{k(N V^2 - C)}}{{\rm\mbox  sign} (k)\cdot V+\tilde C},$$
where constant
$$\tilde C= \pm \frac{\sqrt{k(N V_0^2(x_0)-C)}}{V_0'(x_0)}-{\rm\mbox sign} (k)\cdot V_0(x_0)$$
can be found from the initial data for each $x_0\in \mathbb R$.

Thus, we obtain a connection between the derivatives of solutions of the simple wave type and the solution itself along the characteristic.

\section{Examples}
Using the theorems obtained in Section 3, we find the necessary and sufficient conditions for the singularity formation of the solution of the system \eqref{ep1} in the case $\gamma=0$.
We will consider 4 model cases, in each of them we will construct on the plane $(v_1(0),\, v_2(0))$ a region where solution preserves its smoothness.

1. Let $k=1$, $N=0$, in this case the matrix $Q$ has the form of a Jordan cell: $Q=\left(\begin{smallmatrix}0 & 1 \\ 0 & 0 \end{smallmatrix}\right)$ and, therefore, the transition matrix $A$ coincides with the identity matrix. We will use the statement of Corollary \ref{t6}. Since in this case $C_1=v_1(0)$, $C_2=v_2(0)$, we obtain that the derivatives of the solution tend to infinity if $v_2(0)<0$ or $v_2(0)=0$, $v_1(0)>0$ or $v_2(0)>0$, $v_1(0)<0$, $v_1(0)^2\ge 2 v_2(0)$. The region where we should place the derivatives of the initial data at each point of the real axis in order for the solution to preserve global smoothness is shaded on Fig.~1. The phase portrait of the system \eqref{1} is also represented there. The system has a single equilibrium $B_1 (0,\,0)$. In this case as $t\to\infty$, the functions $E$ and $V$ tend to a constant, which, due to the system \eqref{ep}, can only be zero, the density $n$ tends to zero, and the solutions stabilize to the equilibrium $V=E=0$, $N=0$.

If we consider the case $\gamma>0$, then in addition to the equilibrium $B_1$, we get one more equilbrium $(-\gamma,\,0)$, which is an unstable node.

\begin{figure}[htb!]
\hspace{1.0cm}
\begin{minipage}{0.45\columnwidth}
\includegraphics[scale=0.5]{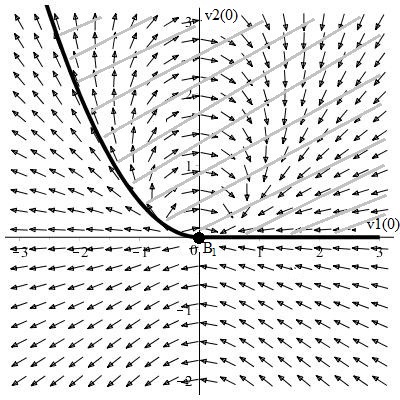}
\end{minipage}
\begin{minipage}{0.45\columnwidth} 
\includegraphics[scale=0.5]{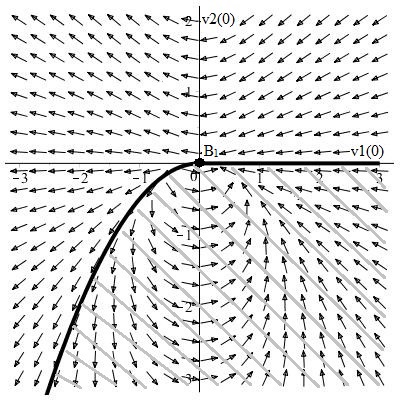}
\end{minipage}
\caption{Domains of attraction for equilibria of the system \eqref{1}. To the left: $k=1$, $N=0$, $\gamma=0$, the part of the parabola $v_1(0)^2=2 v_2(0)$ does not belong to the domain of attraction, $v_1(0)<0$ and the ray $v_2(0)=0$, $v_1(0)\le 0$ belongs to it. To the right: $k=-1$, $N=0$, $\gamma=0$, the part of the parabola $v_1(0)^2=-2 v_2(0)$ does not belong to the domain of attraction, $v_1(0)<0$ and the ray $v_2(0)=0$, $v_1(0)\le 0$ belongs to it.  }
\end{figure}


2. Let $k=-1$, $N=0$, in this case the matrix $Q$ has the form of a Jordan cell: $Q=\left(\begin{smallmatrix}0 & -1 \\ 0 & 0 \end{smallmatrix}\right)$, the transition matrix $A=\left(\begin{smallmatrix}1 & 0 \\ 0 & -1 \end{smallmatrix}\right)$. We will use the statement of Corollary \ref{t6}. Since in this case $C_1=v_1(0)$, $C_2=-v_2(0)$, we obtain that the derivatives of the solution tend to infinity if $v_2(0)>0$ or $v_2(0)=0$, $v_1(0)>0$ or $v_2(0)<0$, $v_1(0)<0$, $v_1(0)^2\ge -2 v_2(0)$. The region where we should the derivatives of the initial data at each point of the real axis in order for the solution to preserve global smoothness is shaded on Fig.~1. The phase portrait of the system \eqref{1} is also represented there; the system has a single sequilibrium $B_1 (0,\,0)$. In this case, as $t\to\infty$, the functions $E$ and $V$ tend to a constant, which, due to the system \eqref{ep}, can only be zero, the density $n$ tends to zero, and the solutions stabilize to the equilibrium $V=E=0$, $N=0$.

If we consider the case $\gamma>0$, then in addition to the equiibrium $B_1$, we have one more equilibrium $(-\gamma,\,0)$ which is an unstable node.


\begin{figure}[htb!]
\hspace{1.0cm}
\begin{minipage}{0.45\columnwidth}
\includegraphics[scale=0.5]{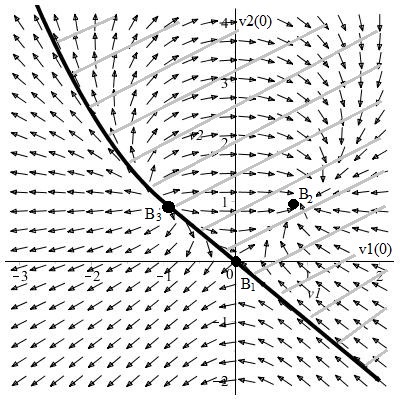}
\end{minipage}
\begin{minipage}{0.45\columnwidth} 
\includegraphics[scale=0.5
]{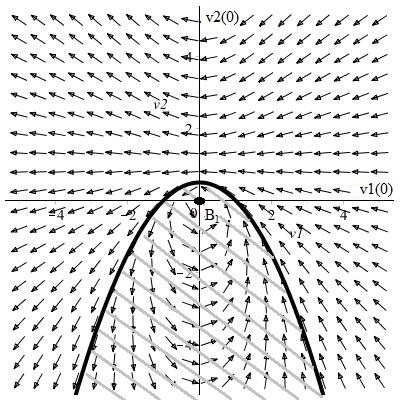}
\end{minipage}
\caption{Domains of attraction for equilibria of the system \eqref{1}. To the left: $k=1$, $N=1$, $\gamma=0$, the domain of attraction does not contain the part of the parabola $2v_2(0)=1+ v_1^2(0)$, $v_1(0)\le -1$ and does contain the ray $v_2(0)=-v_1(0)$, $v_1>-1$. To the right: $k=-1$, $N=1$, $\gamma=0$, the domain of attraction does not contain the parabola $2v_2(0)=1 - v_1^2(0)$.}
\end{figure}

3. Let $k=1$, $N=1$, in this case the matrix $Q$ has the form: $Q=\left(\begin{smallmatrix}0 & 1 \\ 1 & 0 \end{smallmatrix}\right)$, and its eigenvalues equal $\lambda_{1,2}=\pm 1$. The transition matrix is $A=\left(\begin{smallmatrix}1 & 1 \\ 1 & -1 \end{smallmatrix}\right)$. We will use the statement of Corollary \ref{t31}. Since in this case $C_1=\tfrac{1}{2}(v_1(0)+v_2(0))$, $C_2=\tfrac{1}{2}(v_1(0)-v_2(0))$, we obtain that the derivatives tend to infinity in one of the following cases:

\begin{enumerate}
\item $v_1(0)+v_2(0)<0$;

\item $v_1(0)+v_2(0)=0$, $v_1(0)+v_2(0)<-2$;

\item $v_1(0)+v_2(0)>0$, $v_1(0)+v_2(0)<0$, $v_1(0)<0$, $\sqrt{v_2^2(0)-v_1^2(0)} \le v_2(0)-1$.
\end{enumerate}

\noindent The region where we should place the derivatives of the initial data at each point of the real axis in order for the solution to retain global smoothness is shaded on Fig.~2. The phase portrait of the system \eqref{1} is also shown there. In this case, the system has three equlibria $B_1 (0,\,0)$, $B_2 (1,\,1)$, $B_3 (-1,\,1)$. The equilibria $B_1$ and $B_3$ are saddles, $B_2$ is a stable node. In this case, as $t\to\infty$, the functions $E$ and $V$ tend to the affine solution $V=E=x$, the density $n$ tends to zero, and the solutions stabilize at the equilibrium $B_2$.

In the case $\gamma>0$ the system has three equilibria $B_1$ and $\left(\tfrac{-\gamma\pm \sqrt{\gamma^2+4}}{2},\,1\right)$. The last two equilibria are nodes (one is stable, the other one is unstable).


4. Let $k=-1$, $N=1$, in this case the matrix $Q=\left(\begin{smallmatrix}0 & -1 \\ 1 & 0 \end{smallmatrix}\right)$, and its eigenvalues are $\lambda_{1,2}=\pm i$. The transition matrix has the form: $A=\left(\begin{smallmatrix}1 & 0 \\ 0 & -1 \end{smallmatrix}\right)$. We use the statement of Corollary \ref{t7}. Since in this case $C_1=-v_2(0)$, $C_2=v_1(0)$, we obtain that the derivatives tend to infinity if and only if $1-2v_2(0) \le v_1^2(0)$. The region where we place the derivatives of the initial data at each point of the real axis in order for the solution to retain global smoothness is shaded in Fig.~2. The phase portrait of the system \eqref{1} is also shown there; it has periodic solutions. The system has an equilibrium $B_1 (0,\,0)$, which is the center. The behavior of globally smooth solutions is studied in \cite{Ros5}.

In the case $\gamma>0$ equilibrium $B_1$ becomes a focus (for more details see \cite{Ros4}).

\section{Acknowledgements}
The author was supported by Russian Science Foundation (project No.~$23-11-00056$) through RUDN University.







\end{document}